\documentclass[12pt]{article}

\usepackage{amsmath}
\usepackage{amsfonts}
\usepackage{amsthm}
\usepackage{graphicx}
\usepackage{hyperref}
\usepackage{enumerate}
\usepackage[parfill]{parskip}

\numberwithin{equation}{section} 

\newtheorem{thm}{Theorem}[section]
\newtheorem{lem}[thm]{Lemma}

\newtheorem{conj}[thm]{Conjecture}

\begin{document}

\title{Explicit expressions for the moments of the size of an $(s,s+1)$-core partition with distinct parts}

\author{
Anthony Zaleski\thanks{Department of Mathematics, Rutgers University (New Brunswick), 110 Frelinghuysen Road, Piscataway, NJ 08854-8019, USA.}
}

\maketitle
\begin{abstract}
For fixed $s$, the size of an $(s, s+1)$-core partition with distinct parts can be seen as a random variable $X_s$.  Using computer-assisted methods, we derive formulas for the expectation, variance, and higher moments of $X_s$ in terms of $s$.  Our results give good evidence that $X_s$ is asymptotically normal.
\end{abstract}


\\
\\
\emph{A version of this paper was published in \emph{Advances in Applied Mathematics}, Volume 84, March 2017, pp. 1-7.}
%%%%%
%%%%%
%%%%%%%%%%%%
\section{Introduction: the size of an $(s,t)$-core partition}\label{sec:intro}
%\subsection{The size of an $(s,t)$-core partition}

\begin{figure}[ht]
\begin{center}
\includegraphics[width=.5\textwidth]{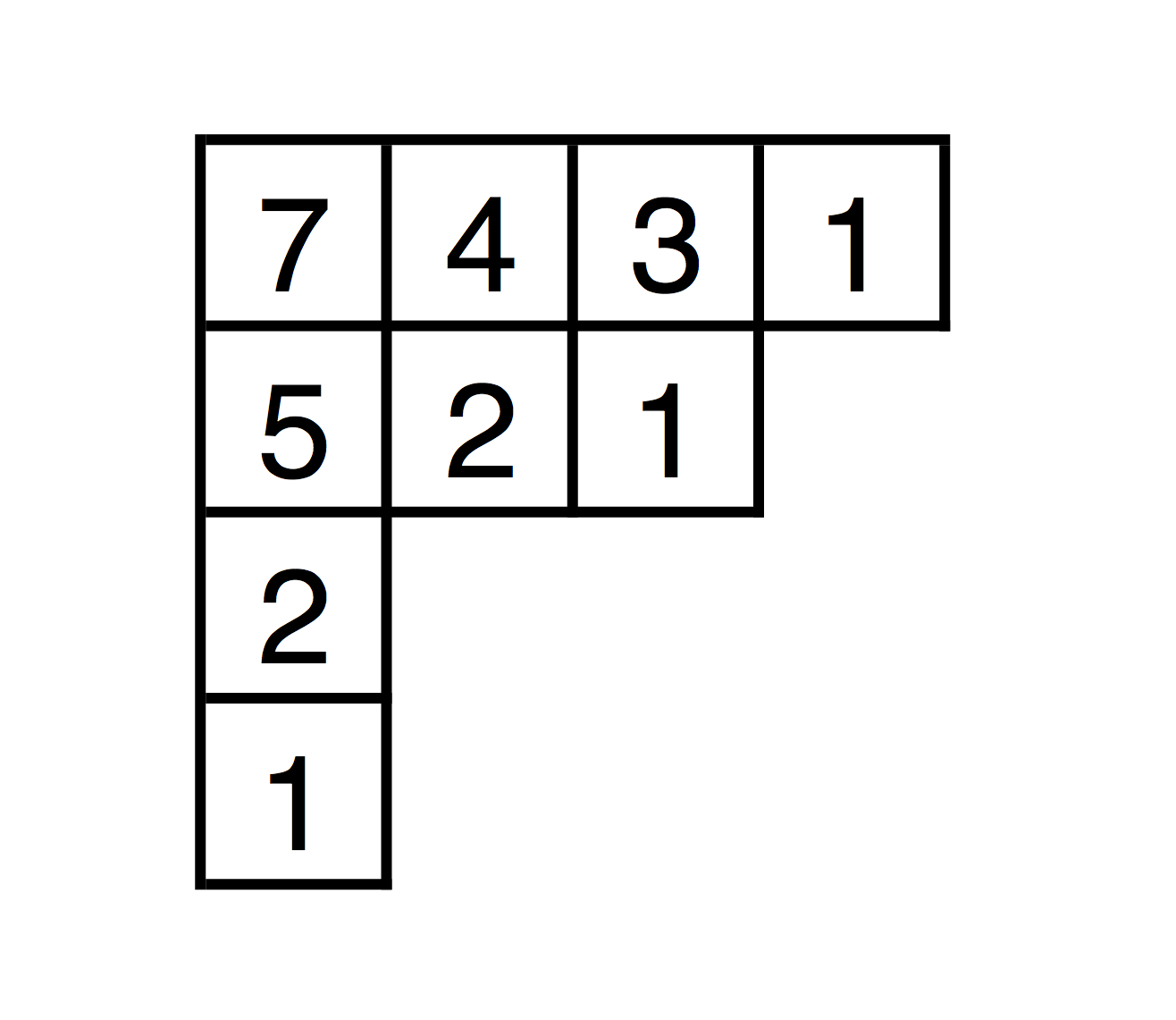}
\caption{Young diagram of the partition $9=4+3+1+1,$ showing the hook lengths of each box.}
\end{center}
\end{figure}

Recall that the \emph{hook length} of a box in the Young diagram of a partition is the number of boxes to the right (the arm) plus the number of boxes below it (the leg) plus one (the head).  (We use the English convention for Young diagrams; see Figure 1.)   A partition is an $s$-\emph{core} if its Young diagram avoids hook length $s$ and an $(s,t)$-\emph{core} if it avoids hook lengths $s$ and $t$ [AHJ].  For example, the partition $9=4+3+1+1$ in Figure 1 is a $(6,8)$-core but not a $(6,7)$-core.

The number of $(s,t)$-core partitions is finite iff $s$ and $t$ are coprime, which we shall assume from now on [AHJ].  Let $X_{s,t}$ be the random variable ``size of an $(s,t)$-core partition," where the sample space is the set of all $(s,t)$-core partitions, equipped with the uniform distribution.  In [EZ], with the help of Maple, Zeilberger derived explicit expressions for the expectation, variance, and numerous higher moments of $X_{s,t}$.  The original paper noted that ``From the `religious-fanatical' viewpoint of the current `mainstream' mathematician, they are `just' conjectures, but nevertheless, they are \textbf{absolutely certain} (well, at least as absolutely certain as most proved theorems)," and a donation to the OEIS was offered for the theory to make the results rigorous.  Later, it was found that such theory did exist and the results are entirely rigorous; see the updates at the paper's site.  

Zeilberger also computed some standardized central moments of $X_{s,t}$ and the limit of these expressions as $s,t\to \infty$ with $s-t$ fixed.  From this he conjectured the limiting distribution.  Perhaps surprisingly, it is abnormal.

Here, we  continue the experimental approach taken up in [EZ].  However, we  consider $(s,t)$-core partitions with \emph{distinct} parts.  Further, we make the restriction $t=s+1$.  Using Maple, we are able to again conjecture, and in this case \emph{rigorously prove} the validity of, explicit expressions for the moments in terms of $s$.  Further, we show that the limiting distribution \emph{does} appear normal in this case.
%%%%
\section{Distinct part $(s,s+1)$-core partitions}
\subsection{Computing the generating function}

Given a positive integer $s$, let $P_s$ be the set of all $(s,s+1)$-core partitions with distinct parts.  Observe that $|P_s|$ is always finite.  Let $X_s$ be the random variable ``size of a partition in $P_s$."  Our goal is to have an efficient way to compute the generating function 
$$
G_s(q):=\sum_{p\in P_s} q^{|p|}
$$
for fixed $s$.  (Here $|p|$ denotes the size of a partition $p$.)  This will then allow us to compute moments of $X_s$.

Recall that the \emph{perimeter} of a partition is the size of the largest hook length.  Straub [S]  gives a useful characterization of $P_s$ in terms of perimeters:
\\
 \begin{lem}[{Lemma 2.2 of [S]}]  \label{lem:perim}
 A partition into distinct parts is an $(s,s+1)$-core iff it has perimeter $<s$.
\end{lem}

From this, Straub also proved Amdeberhan's [A] conjecture that the number of $(s,s+1)$-core partitions with distinct  parts is given by the Fibonacci number: 
\begin{equation}\label{eq:fib}
|P_s|=G_s(1)=F_{s+1}.
\end{equation}

Lemma \ref{lem:perim} gives us a fast way to compute $G_s(q)$.   Define $P_{k,l}$ to be the set of partitions with $l$ distinct parts and largest part $k$.  By the Lemma, a partition $p$ is an $(s,s+1)$-core iff $p \in P_{k,l}$ for some $k+l \leq s$.  Define
$$G_{k,l}(q):=\sum_{p\in P_{k,l}} q^{|p|}.$$
This generating function is computed recursively by \verb+Gkl(q,k,l)+ in the Maple package; see Section 3 for information on the package. Finally, summing $G_{k,l}(q)$ for $k+l\leq s$ gives us $G_s(q)$, implemented in the procedure \verb+Gs(q,s)+ in the Maple package. 
\subsection{Conjecturing moments}

Using \verb+Gs(q,s)+, we can now compute moments of $X_s$ for a given $s$.  Defining the operator $L:f(q)\mapsto qf'(q)$, recall that the $k^\text{th}$ moment of $X_s$ is 
\begin{equation} \label{eq:mom}
\mathbb{E}[X_s^k]= \left. \frac{L^k(G_s(q))}{G_s(q)} \right|_{q=1} = \frac{L^k(G_s(q))(1)}{F_{s+1}},\end{equation}
where we have used \eqref{eq:fib}.

Suppose we fix $k$. Then the numerator, call it $P(s)$, in \eqref{eq:mom} depends only on $s$.  Experimental evidence indicates that $P(s)$ is of the form $A(s)F_s+B(s)F_{s+1}$ for some polynomials $A,B$.  Further, we can use the procedure \verb+GuessFibPol(L,n)+ to guess $A,B$ from computed values of $P(s)$.  

To summarize, we conjecture that for $k$ fixed, there exist polynomials $A(s),B(s)$ such that
\begin{equation} \label{eq:fibpol}
\mathbb{E}[X_s^k]= A(s)\frac{F_s}{F_{s+1}}+B(s),
\end{equation}
and (for fixed $k$), these polynomials can be guessed from  data supplied by \verb+Gs(q,s)+.  

\subsection{Proving the conjectures}

Now we go over the theory needed to validate the above conjectures.  Recall that the $q$-binomial coefficient ${m+n \choose m}_q$ gives the generating function for partitions whose Young diagrams fit inside an $m\times n$ rectangle.  In other words, ${m+n \choose m}_q$ is the sum of $q^{|p|}$, where $p$ ranges over partitions with $\leq m$ parts and largest part $\leq n$. Let us denote these by ``$m \times n$ partitions.''  
\\
\begin{lem}\label{lem:qb}
The generating function (according to size) of partitions $\lambda=(\lambda_1,\dots,\lambda_m)$ with $m$ distinct parts, each  satisfying $\lambda_i\leq n$, is
$$
\sum_{k\leq n} G_{k,m}(q)=q^{m+1 \choose 2} {n \choose m}_q.
$$
Thus,
$$
G_s(q) = \sum_{m=0}^s \sum_{k \leq s-m} G_{k,m}(q)
		   = \sum_{m=0}^s q^{m+1 \choose 2} {s-m \choose m}_q.
$$
\end{lem}

\begin{proof}
Note that ${n \choose m}_q$ is the generating function of $m \times (n-m)$ partitions.   Given such a partition $p$, we can add $1,2,3,\dots,m$ to its parts (counting missing parts as having size 0), producing a partition with exactly $m$ distinct parts of size $\leq n$.  This increases $|p|$ by ${m+1 \choose 2}$.  Further, it is easy to see that this operation defines a bijection.
\end{proof}

Now, since $G_s(q)$ is expressed as a $q$-binomial sum, the theory developed by Wilf and Zeilberger in [WZ] guarantees that $G_s(q)$ satisfies a recurrence. We use the procedure \verb+qGuessRec+ in our Maple package to guess the recursion from the first, say, $30$ terms of the sequence $\{G_s(q)\}_s$, obtaining the following:
%[1,1+q,q^2+q+1,2*q^3+q^2+q+1],[-q^n/q,-q^n/q,0, -1, 1]]
\begin{align}
\begin{split}\label{eq:rec}
G_1(q)&=1
\\
G_2(q)&=1+q
\\
G_3(q)&=q^2+q+1 
\\
G_4(q)&=2q^3+q^2+q+1
\\
G_s(q)&=G_{s-1}(q)+q^{s-1}G_{s-3}(q)+q^{s-1}G_{s-4}(q).
\end{split}
\end{align}

Later, in hindsight, we were able to derive this recurrence  straight from the formula for $G_s(q)$ in Lemma \ref{lem:qb}.  We used Zeilberger's Maple package qEKHAD (see the book [PWZ]), which is capable of both finding and rigorously proving recurrences satisfied by $q$-binomial sums such as the one in our Lemma.  

From \eqref{eq:rec} it follows that the \emph{moments} of the sequence $\{G_s(q)\}_s$ obey the $C$-finite ansatz.  That is, they satisfy linear recurrences with constant coefficients; see [Z2].  Thus, we need only check our conjectures for finitely many values of $s$ to prove them. (In practice, we checked for 70 values of $s$ to compute expressions for up to the sixteenth moment.)

With these observations and the help of Maple, we are now ready to find explicit expressions for the moments of $X_s$.  Fix $k$. We use the recursion \eqref{eq:rec} to efficiently compute the $k^{\text{th}}$ moment of $X_s$ for many values of $s$.  Following Section 2.2, we then conjecture an expression for the $k^{\text{th}}$
moment of $X_s$ which fits the template from  \eqref{eq:fibpol}.  By the above argument, our conjectured expression is proven for all $s$ if it holds for sufficiently many
values of $s$.  
	
%Fix $k$. We use the recursion \eqref{eq:rec} to
%efficiently compute the $k^{\text{th}}$ moment of $X_s$ for many values of $s$.  We  then conjecture an expression for the $k^{\text{th}}$
%moment of $X_s$ which fits the template from \eqref{eq:fibpol} in Section 2.2.  By the previous paragraph, our conjectured expression is proven for all $s$ if it holds for sufficiently many
%values of $s$.  

For moments two and higher, it is more meaningful to compute the \emph{central moment}.  Recall that the $k^\text{th}$  central moment of $X$ is $\mathbb{E}[(X-\mu)^k]$, where $\mu$ is the expectation.  For example, the second central moment is the variance.

Expressions for up to moment 16 may be found in the Maple output file \verb+theorems.txt+. (See Section 3.)  Here is a small sample of the results:
\\
\begin{thm}
Let $X_s$ be the random variable ``size of an $(s,s+1)$-core partition with distinct parts.'' Then,
\begin{enumerate}[(i)]
\item 
$$
\mathbb{E}[X_s]={\frac {1}{50}}\,{\frac {5\,{s}^{2}F_{{s+1}}-6\,sF_{{s}}+7\,sF_{{s+1}}
-6\,F_{{s}}}{F_{{s+1}}}}.
$$
\item
%\begin{align*}
%&(10\,{s}^{3}{\phi}^{2}+20\,{s}^{3}\phi +57\,{
%s}^{2}{\phi}^{2}+33\,{s}^{2}\phi \\ 
%&+65\,s{\phi}^{2}-27\,{s}^{2}-32\,s\phi
%-54\,s-45\,\phi-27) / 1875 \phi^{-2}.
%\end{align*}
\begin{align*}
Var(X_s)&=
\\
&(20\,{s}^{3}F_{{s}}F_{{s+1}}+10\,{s}^{3}{F_{{s+1}}}^{2}-27\,{s}^{2}{F_{
{s}}}^{2}+33\,{s}^{2}F_{{s}}F_{{s+1}}
\\
&+57\,{s}^{2}{F_{{s+1}}}^{2}-54\,s
{F_{{s}}}^{2}-32\,sF_{{s}}F_{{s+1}}+65\,s{F_{{s+1}}}^{2}
\\
&-27\,{F_{{s}}}^{2}-45\,F_{{s}}F_{{s+1}})/(1875F_{s+1}^2).
\end{align*}
\item
The third central moment of $X_s$ is asymptotic to
\begin{align*}
&-(3/31250)(65s^4\phi^3-40s^4\phi^2+222s^3\phi^3-40s^4\phi-218s^3\phi^2
\\&-65s^2\phi^3-106s^3\phi-338s^2\phi^2-390s\phi^3+36s^3-2s^2\phi+110s\phi^2
\\&+108s^2+154s\phi+270\phi^2+108s+90\phi+36)\phi^{-3},
\end{align*}
where $\phi$ is the Golden Ratio. 
\end{enumerate}
\end{thm}

Note that in (iii), we print the asymptotic result simply because the exact expression would take up too much space. Also, (i) is an explicit version of Conjecture 11.9(d) made by Amdeberhan [A] and later proven by Xiong [X].
\subsection{Limiting distribution}

Once we compute the central moments, we can \emph{standardize} them.  Recall that the $k^\text{th}$  standardized central moment of $X$ is $\mathbb{E}[(X-\mu)^k]/\sigma^k$, where $\mu$ is the expectation and $\sigma$ is the standard deviation.  For example, the second standardized central moment is always $1$.  The normal distribution famously has a sequence of standardized central moments which alternates between $0$ and odd factorials: $0,1, 0, 3, 0, 15, 0, 105, 0, 945, 0, 10395, 0, 135135, 0, 2027025,\dots$.

In \verb+theorems.txt+, the limit as $s\to \infty$ of the first 16 standardized central moments of $X_s$ are shown to coincide with that of the normal distribution, giving strong evidence for the following:
\\
\begin{conj}\label{conj:norm}
$X_s$ is asymptotically normal.  That is, the distribution of $(X_s - \mathbb{E}[X_s]) / \sqrt{Var(X_s)}$ converges to the standard normal distribution as $s \to \infty$.
\end{conj}

Note that in [EZ], the limiting distribution of ``size of an $(s,t)$-core partition'' (with the distinct parts condition dropped) was proven to follow an \emph{ab}normal distribution.  

An approach inspired by [Z1] might be useful in proving Conjecture \ref{conj:norm}.  The main idea is to keep track of the \emph{leading terms} in the expressions of the moments, and perhaps use \eqref{eq:rec} to derive a recurrence for the limiting moments.

%%%%%%%%%%
\section{Using the Maple package}
The Maple package \verb+core.txt+ and sample theorems \verb+theorems.txt+ accompanying this paper may be found at the following URL:
 \\ \url{http://www.math.rutgers.edu/~az202/Z}.  
 
 To use the Maple package, place \verb+core.txt+ in the working directory and execute
 %\\
  \verb+read(`core.txt`);+.  
  
  To see the main procedures, execute 
 % \\
  \verb+Help();+. 
  
For help on a specific procedure, use \verb+Help(<procedure name>);+.  
  
%To compute the generating function $G_s(q)$, use \verb+Gs+.  
  
%The theorem-generating procedure, which is used to generate the theorems in \verb+theorems.txt+, is \verb+CoreThms+.  Happy exploring!
%%%%%%%%%%
\section*{Acknowledgement} 
The author thanks Dr. Doron Zeilberger for introducing this project to him and  guiding his research in the right direction.  The author also wishes to thank the two referees for their insightful and
careful reports that helped improve the exposition and the accuracy
of the paper.
%%%%%%%%%%%%%%
\\
\section{References}
\begin{itemize}
\item[{[A]}]
T. Amdeberhan, Theorems, problems, and conjectures (version 20 April, 2016), preprint, 
\\
\url{http://129.81.170.14/~tamdeberhan/conjectures.html}. 
\item[{[AHJ]}]
D. Armstrong, C. R. H. Hanusa, B. C. Jones, Results and conjectures on simultaneous core partitions,  European J. Combin. 41 (2014) 205-220.
%\\
%\url{http://arxiv.org/abs/1308.0572}.
\item[{[EZ]}]
S. B. Ekhad, D. Zeilberger, Explicit expressions for the variance and higher moments of the size of a simultaneous core partition and its limiting distribution, The Personal Journal of Shalosh B. Ekhad and Doron Zeilberger, posted Aug. 30, 2015,
\\
\url{http://www.math.rutgers.edu/~zeilberg/mamarim/mamarimhtml/stcore.html}.
\item[{[PWZ]}]
M. Petkovsek, H. S. Wilf, D. Zeilberger, A=B, A K Peters, Ltd., Wellesley, MA, 1996.
\item[{[S]}]
A. Straub, Core partitions into distinct parts and an analog of Euler's theorem, European J. Combin. 57 (2016) 40-49.
%\\
%\url{http://arxiv.org/abs/1601.07161}
\item[{[WZ]}]
H. S. Wilf, D. Zeilberger,
An algorithmic proof theory for hypergeometric (ordinary and ``q'') multisum/integral identities, Invent. Math. 108 (1992) 575-633.
%\\ 
%\url{http://www.math.rutgers.edu/~zeilberg/mamarimY/Zeilberger_y1992_p575.pdf}
\item[{[X]}]
H. Xiong, Core partitions with distinct parts, preprint, arXiv:1508.07918, 2015.
%\\
%\url{https://arxiv.org/abs/1508.07918}.
\item[{[Z1]}]
D. Zeilberger, The Automatic Central Limit Theorems Generator (and Much More!), in: I. Kotsireas, E. Zima, (Eds.), Advances in Combinatorial Mathematics: Proceedings of the Waterloo Workshop in Computer Algebra 2008, Springer Verlag, Berlin, 2010, pp. 165-174.
%\\
%\url{http://www.math.rutgers.edu/~zeilberg/mamarim/mamarimhtml/georgy. html}.
\item[{[Z2]}]
D. Zeilberger, The C-finite Ansatz,
Ramanujan Journal 31(2013) 23-3.
%\\
%\url{http://www.math.rutgers.edu/~zeilberg/mamarim/mamarimhtml/cfinite.html}
%\item[{[Z3]}]
%D. Zeilberger, qEKHAD, Maple package,  available at the following URLs:
%\\
%\url{http://www.math.rutgers.edu/~zeilberg/tokhniot/qEKHAD}
%\\
%\url{http://www.math.rutgers.edu/~zeilberg/programsAB.html}.
\end{itemize}
%\bibliography{references}
%\bibliographystyle{plain}

\end{document}